\documentclass[11pt]{article}

\usepackage[paperwidth=8.5in,paperheight=11in,portrait,top=1in,bottom=1.in,left=1.15in,right=1.15in]{geometry}
\usepackage[utf8]{inputenc}
\usepackage{authblk}

\usepackage{amsfonts,amsmath,amssymb,amsthm}
\usepackage{latexsym,mathrsfs,mathtools,bm}
\usepackage{graphicx,subcaption,epsfig,caption,float,xcolor}
\usepackage{enumitem}
\usepackage{chngcntr}

\usepackage{tikz}
\usetikzlibrary{calc}

\usepackage[hidelinks]{hyperref}
\usepackage{bookmark}

\newcommand{\End}{\text{End}}

\newcommand{\cB}{\mathcal{B}}

\newcommand{\gsl}{\mathfrak{sl}}

\newcommand{\osp}{\mathfrak{osp}}

\newcommand{\s}{s}
\newcommand{\x}{x}
\newcommand{\AWB}[1][\s]{AW\!B_3^{#1}(q)}
\newcommand{\A}{\mathcal{A}_s}
\newcommand{\Z}{Z_\s}
\newcommand{\ZU}{\End_{\Usl}(V_s^{\otimes 3})}

\newcommand{\qi}{{q^{-1}}}
\newcommand{\Usl}{U_q(\gsl_2)}
\newcommand{\Ut}{U_q(\gsl_2)^{\otimes 3}}
\newcommand{\Vt}{V_s^{\otimes 3}}

\theoremstyle{plain}
\newtheorem{thm}{Theorem}[section]

\newtheorem{prop}[thm]{Proposition}
\newtheorem{coro}[thm]{Corollary}

\newtheorem{defi}[thm]{Definition}
\newtheorem{rem}[thm]{Remark}

\numberwithin{equation}{section}

\def\lw{0.7pt} 
\def\xyd{0.8cm} 

\newcommand{\bul}{
	circle(0.06cm)
} 

\newcommand{\spath}{
	(\xyd,0)  \bul
} 

\newcommand{\upath}{
	(\xyd,\xyd) \bul
} 

\newcommand{\dpath}{
	(\xyd,-\xyd) \bul
} 

\title{\bf Askey--Wilson braid algebra and centralizer of $U_q(\mathfrak{sl}_2)$}

\renewcommand*{\Affilfont}{\normalsize\small}
\author[1]{Nicolas Cramp\'e}
\author[2]{Lo\"ic Poulain d'Andecy}
\author[3]{Luc Vinet}
\author[4]{Meri Zaimi\vspace{.5em}}

\affil[1]{Institut Denis-Poisson CNRS/UMR 7013 - Universit\'e de Tours - Universit\'e d'Orl\'eans, \newline\vspace{.9em}
Parc de Grandmont, 37200 Tours, France.}
\affil[2]{Laboratoire de math\'ematiques de Reims UMR 9008, Universit\'e de Reims Champagne-Ardenne, \newline\vspace{.9em} 
Moulin de la Housse BP 1039, 51100 Reims, France.}
\affil[2]{IRL-CRM, CNRS UMI 3457, Universit\'e de Montr\'eal. \vspace{.9em}}
\affil[3,4]{Centre de Recherches Math\'ematiques, Universit\'e de Montr\'eal, \newline\vspace{.9em}
P.O. Box 6128, Centre-ville Station, Montr\'eal (Qu\'ebec), H3C 3J7, Canada.}
\affil[3]{Insitut de valorisation des donn\'ees (IVADO), Montr\'eal (Qu\'ebec), H2S 3H1, Canada. \newline\vspace{.9em}}

{
	\makeatletter
	\renewcommand\AB@affilsepx{: \protect\Affilfont}
	\makeatother
	\affil[ ]{E-mail addresses}
	\makeatletter
	\renewcommand\AB@affilsepx{, \protect\Affilfont}
	\makeatother
	\affil[1]{crampe1977@gmail.com}
	\affil[2]{loic.poulain-dandecy@univ-reims.fr}
	\affil[3]{vinet@crm.umontreal.ca}
	\affil[4]{meri.zaimi@umontreal.ca}
}

\begin{document}
	
\date{\today} 
\maketitle

\begin{abstract}
\noindent A presentation of the centralizer of the three-fold tensor product of the spin $s$ representation of the quantum group $\Usl$ is provided. It is expressed as a quotient of the Askey--Wilson braid algebra. This newly defined algebra combines the Askey--Wilson relations with the braid group relations, on three strands, together with a characteristic equation of degree $2s+1$ for the braid generators. Explicit bases are given for the centralizer.
\end{abstract}

\maketitle

\vspace{3mm}

\section{Introduction}


Numerous quotients of the braid group algebra have been studied previously and occur in different contexts. For example, they naturally appear when considering the centralizer of tensor products of quantum group representations, and are used in the computation of the invariants for knot theory \cite{Jo85,Jo87,BW,Mur}. Note that explicit algebraic presentations for these centralizers are known only in a few cases. The most basic example is the Hecke algebra, which corresponds to a quotient of the braid group algebra realized by demanding that the generators satisfy a polynomial equation of degree $2$. The centralizer of $U_q(\gsl_n)$ in the tensor product of its fundamental representation is in general isomorphic to a quotient of the Hecke algebra \cite{Jimbo,Res}. For $n=2$, this centralizer is the Temperley--Lieb algebra \cite{TL}. 
A quotient of the braid group algebra by a polynomial equation of degree $3$ and additional relations, called the Birman--Murakami--Wenzl algebra \cite{BW,Mur}, has also been studied intensively. It is related to the centralizer of $\Usl$ in the spin $1$ representation \cite{LZ}. Generalizations of the Hecke algebra for any symmetric representation have been described recently in \cite{CP}.
Let us also mention that quotients of the braid group algebra by a polynomial equation of degree $k$ for the braid generators appear in the world of complex reflection groups \cite{BM,BMR,Cha,Mar}. 
In particular, the resulting algebra is finite-dimensional only when $\frac{1}{k}+\frac{1}{n}>\frac{1}{2}$ where $n$ is the number of strands. It turns out to be quite difficult to find presentations of finite-dimensional quotients of the braid group algebra, see \emph{e.g.}\ \cite{Mar2} for the quotients by a cubic characteristic relation for the generators.
The main result of this paper can be seen as a new family of explicit presentations for finite-dimensional quotients of the braid group algebra on three strands.
For each integer $k>1$, one member of this family has a characteristic equation of degree $k$, and the additional relations to make it finite-dimensional (and in fact, isomorphic to the centralizer) are obtained from the Askey--Wilson algebra.

A new point of view on the centralizers of $U(\gsl_2)$ or $\Usl$ in tensor products of any three possibly different spin representations has been brought up in \cite{CPV,CVZ2}.
Indeed, these centralizers are related to the Racah or Askey-Wilson algebras.
The Racah algebra has been introduced in \cite{GZ} in connection with the $6j$-symbols intervening in the coupling of three angular momenta.
The $q$-deformation of this approach has been initiated in \cite{Zh} where the Askey--Wilson algebra is introduced.
The Racah and Askey--Wilson algebras have subsequently found applications in different contexts such as algebraic combinatorics \cite{GWH,Ter}, Kauffman bracket skein algebras \cite{CL}
and symmetry of physical models \cite{KKM,Post,GVZ}.
Their representation theory has been studied in \cite{HB,Hua2,Hua}. A review on the Racah algebra and its generalizations as well as the connection with the universal diagonal centralizer of $U(\gsl_2)$ can be found in \cite{CGPV}, and a review of the Askey--Wilson algebra is given in \cite{CFGPRV}.

\medskip
The goal of this paper is to combine the braid relations with the Askey--Wilson relations in order to provide 
a description in terms of generators and relations of the centralizer of $\Usl$ in the tensor product of three identical spin $s$ representations. Our focus on three-fold tensor products (the three-strands case) is justified by the expectation that 
in general the $n$-fold centralizers will be defined by relations involving a limited number of strands (this is called locality of the presentation). In other words, we expect the centralizer for general $n$ to be fully described once a finite number of small values of $n$ are understood. The level $n=3$ is the first non-trivial one, where for example the braid relation appears. This level is enough to define the Temperley--Lieb, Hecke and Birman--Murakami--Wenzl algebras. In our case we expect to have new relations appearing at higher level, but nevertheless a good understanding of the case $n=3$ is necessary for the general case.

To achieve our goal, we start with the quotient of the three-strands braid group algebra by a characteristic polynomial equation of degree $2s+1$. This quotient is a generalization of the three-strands Hecke algebra, which is recovered when $s=1/2$. 
A new algebra, called Askey--Wilson braid algebra $\AWB$, is then defined by adding the Askey--Wilson relations. This algebra can both be seen as a quotient of the three-strands braid group algebra and as a quotient of the usual Askey--Wilson algebra.
The remaining of the paper consists in the study of the quotient of the Askey--Wilson braid algebra by two supplementary relations which are enough to prove that it is isomorphic to the centralizer of $\Usl$ in the tensor product of three spin $s$ representations. Let us remark that the connection between the braid group and the Askey--Wilson algebra has already been exploited to give a simpler interpretation of some generators of the Askey--Wilson algebra \cite{CGVZ} and to find representations of the braid relation in terms of the $q$-Racah polynomials \cite{CVZ3}. The latter point is related to spin Leonard pairs \cite{Curtin,NT}.

The plan of this paper is as follows. Section \ref{sec:defAWB} provides the notations and contains definitions of different quotients of the braid group algebra on three strands.
In particular the quotient $\AWB$, the main object of interest of this paper, is given in Definition \ref{def:AWB}. 
Then, the isomorphism between a quotient of $\AWB$ and the centralizer of $\Usl$ in the tensor product of three identical representations is stated in Theorem \ref{thm:iso} and the rest of the paper is mainly devoted to its proof.
The definition of the centralizer of interest in this paper is given in Section \ref{sec:cent}. 
Using previous results in the literature, the surjectivity is established.
Section \ref{sec:iso} focuses on the proof of the injectivity. Two explicit bases of the quotient of $\AWB$ are provided. The cases where $s=1/2$ and $s=1$ are specifically considered as examples, and the connections with the Temperley--Lieb and Birman--Murakami--Wenzl algebras are exhibited.
The paper concludes with some perspectives in Section \ref{eq:conc}. Technical proofs are gathered in appendix \ref{app}.

\section{The Askey--Wilson braid algebra and its special quotient \label{sec:defAWB}}

Let $\s \in \mathbb{Z}_{>0}/2$ be a strictly positive integer or half-integer and $q$ be an indeterminate. Throughout the paper, we will use the $q$-commutator of two elements:
$$[A,B]_q=qAB-q^{-1}BA,$$
and the following notations:
\begin{equation}
	q_p:=(-1)^pq^{p(p+1)} \quad \text{and} \quad \chi_p:=q^{2p+1}+q^{-2p-1} .\label{eq:qpetchip}
\end{equation}

\subsection{Preliminaries on Askey--Wilson algebras} 
We shall consider a specialization  of the Askey--Wilson algebra $\mathfrak{\textbf{aw}}(3)$ (see \cite{CFGPRV} for a review). 
It is generated by three generators $A$, $B$ and $C$ and demanding that the following three elements 
\begin{equation} 
	\frac{1}{q^2-q^{-2}} [A,B]_q+ C ,\qquad  \frac{1}{q^2-q^{-2}} [C,A]_q+ B, \qquad \frac{1}{q^2-q^{-2}} [B,C]_q+ A,  \label{eq:ABC}
\end{equation}
are equal to the same central element written as follows for later convenience  \[\frac{\chi_s(\kappa+\chi_s)}{q+\qi} .\] 
In this specialisation of $\mathfrak{\textbf{aw}}(3)$, the central element (computed in \cite{Zh}) reads:
\[ \Omega=\chi_s(\kappa+\chi_s)(q A+q^{-1}B+qC)-q^2 A^2-q^{-2}B^2 -q^2C^2 -q ABC . 
\]
A quotient of $\mathfrak{\textbf{aw}}(3)$ is obtained by fixing the values of this central element $\Omega$:
\[\Omega=\kappa(\kappa+\chi_s^3)+3\chi_s^2-(q+q^{-1}) . \label{eq:saw4}\]
This quotient, called special Askey--Wilson algebra, is important since it allows to provide a precise connection with the centralizers of $\Usl$ representations (for more details, see \emph{e.g.}\ \cite{CFGPRV}).

\subsection{A quotient of the braid group algebra}

We define a quotient of the braid group algebra on three strands by imposing a characteristic equation of order $2\s+1$ for the braid generators.

The algebra $B^\s_3(q)$ is generated by $\sigma_1$ and $\sigma_2$ subject to the following defining relations:
\begin{align}	
	& \sigma_1\sigma_2\sigma_1=\sigma_2\sigma_1\sigma_2, \label{eq:braid} \\
	& \prod_{p=0}^{2\s} (\sigma_i -q_p   ) =0 ,\quad\text{for } i=1,2. \label{eq:chareqsig}
\end{align}
In $B^\s_3(q)$, we define the idempotents, for $r=0,1,\dots,2\s$ and $i=1,2$,
\begin{equation}
	\displaystyle e_i^{(r)}:= \prod_{\genfrac{}{}{0pt}{}{p=0}{p\neq r} }^{2\s} \frac{\sigma_i-q_p}{q_r-q_p} \label{eq:eir}
\end{equation}
satisfying 
\begin{equation}
	\sum_{r=0}^{2\s} e_i^{(r)}=1 , \quad e_i^{(p)}e_i^{(r)}=\delta_{pr}e_i^{(r)} \quad \text{and} \quad e_i^{(r)} \sigma_i =\sigma_i e_i^{(r)}  = q_r e_i^{(r)}. \label{eq:epropr}
\end{equation}
The generators $\sigma_i$ can be written in terms of these idempotents as
\begin{equation}
	\sigma_i = \sum_{r=0}^{2\s} q_r e_i^{(r)} . \label{eq:sige}
\end{equation}

\subsection{The Askey--Wilson braid algebra}

We are ready to define an algebra combining the algebra $B^\s_3(q)$ and the Askey--Wilson relations. We define in $B^\s_3(q)$ the elements
\begin{equation}
	g_i :=\sum_{r=0}^{2\s} \chi_{r}e_i^{(r)} , \quad\text{for}\ i=1,2 , \label{eq:gi}
\end{equation}
which allows to give the following definition: 
\begin{defi}\label{def:AWB} Let $\s \in \mathbb{Z}_{>0}/2$.
	The Askey--Wilson braid algebra (on three strands) $\AWB$ is the quotient of the algebra $B^\s_3(q)$ obtained by asking that
	the triple $(g_1,g_2,\sigma_1 g_2 \sigma_1^{-1})$ satisfies the relations \eqref{eq:ABC} of the Askey--Wilson algebra, \textit{i.e.}\
	the three elements 
	\begin{equation}\label{AW-triple}
		\frac{[g_1,g_2]_q}{q^2-q^{-2}} + \sigma_1 g_2 \sigma_1^{-1} , \qquad \frac{[g_2,\sigma_1 g_2 \sigma_1^{-1}]_q}{q^2-q^{-2}} + g_1 , \qquad \frac{[\sigma_1 g_2 \sigma_1^{-1},g_1]_q}{q^2-q^{-2}}+ g_2 ,
	\end{equation}
	are equal to the same central element which will still be denoted by \[\frac{\chi_s(\kappa+\chi_s)}{q+\qi} .\] 
\end{defi}
Since the relations of the Askey--Wilson algebra are satisfied in $\AWB$, the element
\begin{equation}
	\Omega:=\chi_\s(\kappa+\chi_\s)(qg_1+q^{-1} g_2+q\sigma_1 g_2\sigma_1^{-1})-q^2 g_1^2-q^{-2} g_2^2 -q^2 \sigma_1 g_2^2 \sigma_1^{-1} -q g_1 g_2\sigma_1 g_2\sigma_1^{-1} \label{eq:Om}
\end{equation}
is also central in $\AWB$.
The main motivation for the previous definition is the connection with the centralizer of $\Usl$ given in the following theorem.
\begin{thm}\label{thm:iso}
	The quotient $\A$ of $\AWB$ by relations
	\begin{align}  
		& \Omega = \kappa (\kappa + \chi_\s^3) + 3 \chi_\s^2 - \chi_0^2 , \label{eq:AWB1b} \\
		& e_1^{(0)}g_2e_1^{(0)} = \frac{\chi_\s^2}{\chi_0}e_1^{(0)}, \label{eq:AWB2b}
	\end{align}
	is isomorphic to the centralizer of the diagonal action of $\Usl$ in three copies of its spin $s$ irreducible representation. 
\end{thm}
The proof of this theorem is given in Sections \ref{sec:cent} and \ref{sec:iso}.
One recognizes with (\ref{eq:AWB1b}) the known relation defining the special Askey--Wilson algebra, so it is not surprising to find it here. The quotient of $\AWB$ only by \eqref{eq:AWB1b}
could be called the special Askey--Wilson braid algebra.
\begin{rem}
	It is an open question whether the Askey--Wilson braid algebra $\AWB$ is finite-dimensional. What one can show is that, adapting slightly the proofs in Section \ref{sec:iso}, the algebra becomes finite-dimensional after adding the relation \eqref{eq:AWB1b}. Note that we need the final relation \eqref{eq:AWB2b} to prove the isomorphism with the centralizer but we conjecture that it is actually implied by the other relations. We were not able to prove this, but it seems to be true for low values of the spin $s$ ($s=1/2,1,3/2,2$), see Remark \refeq{rem:relkape0} for additional comments on this.
\end{rem}

Now a remarkable simplification occurs. 
Indeed it is easy to check that the three elements appearing in \eqref{AW-triple} are conjugated by $\sigma_1\sigma_2$ in $B^\s_3(q)$ using 
\begin{equation}\label{rel-si-g}
	[\sigma_i,g_i]=0,\ \qquad \sigma_1\sigma_2g_1=g_2\sigma_1\sigma_2,\ \qquad\sigma_2\sigma_1g_2=g_1\sigma_2\sigma_1.
\end{equation}
Therefore, the statement that the three elements \eqref{AW-triple} are central can be replaced equivalently by demanding that the element 
\begin{equation}
	\kappa:=\frac{1}{(q-\qi)\chi_s}[g_1,g_2]_q+\frac{q+\qi}{\chi_s}\sigma_1 g_2 \sigma_1^{-1} -\chi_s\ \ \ \in\cB^\s_3(q) \label{eq:kappa}
\end{equation}
is a central element. This allows to give an equivalent presentation of $\AWB$ but also of its quotient $\A$ as summarized in the following corollary:
\begin{coro}\label{cor:AWB} The algebra $\A$ is isomorphic to the quotient of $\cB^\s_3(q)$ by
	\begin{align}
		& \kappa \text{ is central,} \label{eq:AWB0}\\		
		& \Omega = \kappa (\kappa + \chi_\s^3) + 3 \chi_\s^2 - \chi_0^2 , \label{eq:AWB1} \\
		& e_1^{(0)}g_2e_1^{(0)} = \frac{\chi_\s^2}{\chi_0}e_1^{(0)}. \label{eq:AWB2}
	\end{align}
\end{coro}

Let us remark that the algebra $\AWB$ is also generated by $g_1$ and $g_2$ (instead of $\sigma_1$ and $\sigma_2$). Indeed, these elements satisfy
\begin{equation}
	e_i^{(r)} g_i =g_i e_i^{(r)}  = \chi_{r} e_i^{(r)} \label{eq:ge}
\end{equation}
and the characteristic equation
\begin{equation}
	\prod_{p=0}^{2\s} (g_i -\chi_{p}) =0 . \label{eq:chareqg}
\end{equation}
As a consequence, the idempotents $e_i^{(r)}$ can be expressed in terms of $g_i$ as follows
\begin{equation}
	e_i^{(r)}=\prod_{\genfrac{}{}{0pt}{}{p=0}{p\neq r}}^{2\s} \frac{g_i-\chi_p}{\chi_r-\chi_p} , \label{eq:eprodg}
\end{equation}
which shows that $g_1$ and $g_2$ generate $\AWB$. Therefore, $\AWB$ could equivalently be defined as a quotient of the Askey--Wilson algebra by the characteristic polynomial \eqref{eq:chareqg} 
and the braid relation \eqref{eq:braid}. 

\section{Centralizer of $\Usl$ \label{sec:cent}}

In this section, we give an explicit surjective homomorphism from the algebra $\A$ to the centralizer of $\Usl$ defined more precisely below. In other words, we show that some elements of the centralizer of $\Usl$ satisfy the relations of the Askey--Wilson braid algebra $\AWB$ but also of its quotient $\A$. We start by recalling well-known results about $\Usl$.

\subsection{Tensor products of representations of $\Usl$}

The quantum group $\Usl$ is a quasitriangular Hopf algebra that can be viewed as a $q$-deformation of the universal enveloping algebra of $\gsl_2$. Its definition by generators and relations can be found in \cite{CVZ3} and will not be reproduced here. Let us recall that the deformation parameter $q$ is taken to be generic in this paper.  

For each non-negative integer or half-integer $\s$ referred to as the spin, the algebra $\Usl$ has a finite irreducible representation $V_{\s}$ of dimension $2\s+1$. The tensor product of two such representations decomposes as follows
\begin{equation}
	V_s^{\otimes 2} \cong \bigoplus_{j=0}^{2s}V_j . \label{eq:decompV2}
\end{equation}
Similarly, the tensor product of three copies of $V_s$ decomposes as
\begin{equation}
	V_s^{\otimes 3}\cong \bigoplus_{j=0}^{2s}V_j\otimes V_s \cong \bigoplus_{j=j_\text{min}}^{3s} V_j^{\oplus d_j} ,
	 \label{eq:decompV3}
\end{equation}
where
\begin{equation}
	j_\text{min} = 
	\begin{cases}
		0 , & \text{if $s$ is integer}, \\
		\frac{1}{2} , & \text{if $s$ is half-integer},
	\end{cases} \label{eq:jmin}
\end{equation}
and the degeneracy $d_j$ of $V_j$ is given explicitly by
\begin{equation}
	d_j=\min(2s,s+j) - |s-j| + 1 = 
	\begin{cases}
		2j+1 ,  & \text{if } j_{\min} \leq j\leq s ,\\
		3s-j+1 , & \text{if } s<j\leq 3s .\\
	\end{cases} \label{eq:degen}
\end{equation}

\subsection{Centralizer}
If $X \in \Usl$, we will denote by $X_{\vert_{\Vt}}$ its representation in $\End(\Vt)$. The algebra of interest is the centralizer of $\Usl$ in $\End(\Vt)$ defined as
\begin{equation}
	\Z := \ZU = \{M \in \End(V_s^{\otimes 3}) \ | \ MX_{\vert_{\Vt}}=X_{\vert_{\Vt}}M \quad \forall X \in \Usl\}. \label{eq:centr}
\end{equation}  
Because of the first isomorphism in \eqref{eq:decompV3}, there is a vector space isomorphism
\begin{equation}
	\Z \cong \bigoplus_{a,b=0}^{2s}\text{Hom}_{\Usl}\bigl(V_a\otimes V_s \,, \, V_{b}\otimes V_s\bigr),
\end{equation}
where
\begin{align}
	&\text{Hom}_{\Usl}\bigl(V_a\otimes V_s\,,\,V_{b}\otimes V_s\bigr) \nonumber \\ 
	&=\{M \in \text{Hom}\bigl(V_a\otimes V_s\,,\,V_{b}\otimes V_s\bigr) \ | \ MX_{\vert_{V_a\otimes V_s}}=X_{\vert_{V_b\otimes V_s}}M \quad \forall X \in \Usl\}. \label{eq:centr2}
\end{align} 

An explicit formula for the dimension of the centralizer $\Z$ for any choice of spin $s$ can be obtained using the degeneracies $d_j$ given in \eqref{eq:degen}:
\begin{equation}
	\dim(\Z)=\sum_{j=j_\text{min}}^{3s}d_j^2 = \frac{1}{2}(2\s + 1)( (2\s+1)^2 +1) . \label{eq:dimcentr}
\end{equation} 
Another way to compute the dimension of the centralizer is to use the vector space direct sum decomposition \eqref{eq:centr2}. The dimension of the summand corresponding to $a,b$ in this direct sum is
\begin{equation}
	\min(a+b,2s) - |a-b| + 1 
	=\begin{cases} 
		2\min(a,b)+1, & \text{if $a+b \leq 2s$},\\
		2s-|a-b|+1, & \text{if $a+b>2s$}.
	\end{cases} \label{eq:dimcentr2}
\end{equation}
The dimension of the centralizer $\Z$ is then the sum of the previous dimensions over $a,b=0,1, \dots, 2\s$, which coincides with \eqref{eq:dimcentr}.

\subsection{Elements and relations for the centralizer}
There is a quadratic Casimir element in $\Usl$ that can be normalized (see \cite{CVZ3}) such that it is represented in $\End(V_s)$ by
\begin{equation}
	\chi_s \text{Id}_{V_s}, \label{eq:repCasimir}
\end{equation}
where $\chi_s$ is defined in \eqref{eq:qpetchip} and $\text{Id}_{V_s}$ denotes the identity acting on $V_s$. The normalization choice \eqref{eq:repCasimir} implies that in the direct sum basis \eqref{eq:decompV2}, the Casimir element of $\Usl$ is represented by the following direct sum matrix in $\End(V_s^{\otimes 2})$
\begin{equation}
	C  = \bigoplus_{j=0}^{2\s} \chi_j \text{Id}_{V_j}. \label{eq:repCasimir2}
\end{equation} 
Since the Casimir element is central in $\Usl$, its representation \eqref{eq:repCasimir2} belongs to the centralizer of $\Usl$ in $\End(V_s^{\otimes 2})$. 

There is another element $\check R$ in this centralizer that can be constructed using the universal $R$-matrix of $\Usl$. The matrix $\check R$ takes the following form in the decomposed basis \eqref{eq:decompV2}
\begin{equation}
	\check R = \alpha_s \bigoplus_{j=0}^{2\s} q_j \text{Id}_{V_j}, \label{eq:cR}
\end{equation}
where $q_j$ is defined in \eqref{eq:qpetchip} and $\alpha_s = (-1)^{2\s} q^{-2\s(\s+1)}$ is the normalization constant that appears when using the conventions of \cite{CVZ3} (a detailed proof of this result can be found therein). Equations \eqref{eq:repCasimir2} and \eqref{eq:cR} imply that we can define the following orthogonal idempotent matrices for $r=0,1,\dots,2\s$
\begin{equation}
	\displaystyle E^{(r)} := \prod_{\genfrac{}{}{0pt}{}{p=0}{p\neq r} }^{2\s} \frac{\alpha_s^{-1}\check R-q_p \text{Id}_{V_s^{\otimes 2}}}{q_r-q_p} 
	= \prod_{\genfrac{}{}{0pt}{}{p=0}{p\neq r} }^{2\s} \frac{C-\chi_p \text{Id}_{V_s^{\otimes 2}}}{\chi_r-\chi_p} \label{eq:Er}
\end{equation} 
which satisfy
\begin{equation}
	E^{(r)}\check R = \check R E^{(r)} = \alpha_sq_r E^{(r)}, \quad E^{(r)}C = CE^{(r)} = \chi_r E^{(r)}. \label{eq:proprEr}
\end{equation}
In other words, $E^{(r)}$ is the projector on the irreducible representation space $V_r$ in the direct sum decomposition \eqref{eq:decompV2}. We can write
\begin{equation}
	C = \sum_{r=0}^{2s} \chi_r E^{(r)}. \label{eq:CEr}
\end{equation}
 
Now we consider the three-fold tensor product $V_s^{\otimes 3}$. We denote $C_{12}:=C \otimes \text{Id}_{V_s}$ and $C_{23}:=\text{Id}_{V_s} \otimes C$. The matrices $\check R_{12},\check R_{23}$ and $E^{(r)}_{12},E^{(r)}_{23}$ are understood 
similarly using \eqref{eq:cR} and \eqref{eq:Er} respectively. It is then possible to define
\begin{equation}
	C_{13} := \check R_{12} C_{23} \check R_{12}^{-1}. \label{eq:C13} 
\end{equation}
We also denote by $C_{123}$ the matrix representing the action of the Casimir element of $\Usl$ on $V_s^{\otimes 3}$ which takes the following form in the fully decomposed basis of \eqref{eq:decompV3}:
\begin{equation}
	C_{123} = \bigoplus_{j=j_\text{min}}^{3\s} \chi_j \text{Id}_{V_j}^{\oplus d_j}. \label{eq:repCasimir3}
\end{equation} 

The triple of matrices $C_{12},C_{23},C_{13}$ are representations of the so-called intermediate Casimir elements of $\Ut$ in $\End(V_s^{\otimes 3})$.
With the normalization \eqref{eq:repCasimir}, they satisfy the relations of the special Askey--Wilson algebra as defined in Section \ref{sec:defAWB}, with the total Casimir element $C_{123}$ 
corresponding to the central element $\kappa$ \cite{Zh,CGVZ}. For example, one gets 
\begin{equation}
	\frac{[C_{12},C_{23}]_q}{q^2-q^{-2}} + \check R_{12} C_{23} \check R_{12}^{-1} =  \frac{\chi_s}{\chi_0} (C_{123}+\chi_s \text{Id}_{V_s^{\otimes 3}}). \label{eq:AWCasimirs}
\end{equation}
The matrices $C_{12},C_{23},C_{13},C_{123}$ all belong to the centralizer $\Z$, 
as can be deduced from the discussions above 
(for the element $C_{13}$, this fact is a direct consequence of the results of \cite{CGVZ}). In \cite{CVZ2}, it has been proven 
that they generate the whole centralizer $\Z$.
It is also apparent that the matrices $\check R_{12}$ and $\check R_{23}$ belong to $\Z$ and they generate the centralizer $\Z$ \cite{LZ2}. 
They are called braided $R$-matrices because they satisfy the braided Yang--Baxter equation
\begin{equation}
	\check R_{12}\check R_{23}\check R_{12} = \check R_{23} \check R_{12} \check R_{23}. \label{eq:braidR}
\end{equation}

The total Casimir element $C_{123}$ has the property
\begin{equation}
	E_{12}^{(0)}C_{123}E_{12}^{(0)} = C_{123}E_{12}^{(0)} = \chi_s E_{12}^{(0)}. \label{eq:C123E0}
\end{equation}
Indeed, recall that $E_{12}^{(0)}$ is the projector on the space $V_0 \otimes V_s$ of the first decomposition in \eqref{eq:decompV3}. 
The matrix $E_{12}^{(0)}C_{123}E_{12}^{(0)}$ hence acts as zero on all the summands of this decomposition except on $V_0 \otimes V_s \cong V_s$ 
where it acts as the constant $\chi_s$. This leads directly to equation \eqref{eq:C123E0}. 
Multiplying \eqref{eq:AWCasimirs} on the left and on the right by the idempotent $E_{12}^{(0)}$ and using the properties \eqref{eq:proprEr} together with \eqref{eq:C123E0}, it is straightforward to find
\begin{equation}
	E_{12}^{(0)}C_{23}E_{12}^{(0)} = \frac{\chi_\s^2}{\chi_0}E_{12}^{(0)}.
\end{equation}
The discussion above leads to the following proposition. 
\begin{prop}\label{prop:surj} The mapping given by
\begin{align}
	 &\sigma_1 \mapsto \alpha_s^{-1} \check R_{12},\qquad  \sigma_2 \mapsto \alpha_s^{-1} \check R_{23}, \label{eq:map_surj}
\end{align}
extends to a surjective algebra homomorphism from $\A$ to the centralizer $\Z$ such that
\begin{align}
	\ g_1 \mapsto C_{12}, \qquad g_2 \mapsto C_{23},  \qquad  \sigma_1 g_2 \sigma_1^{-1} \mapsto C_{13}, \qquad  \kappa \mapsto C_{123}.\label{eq:mapsurjg}
\end{align}
\end{prop}

\section{Isomorphism with the centralizer\label{sec:iso}}

This section is devoted to the proof of Theorem \ref{thm:iso} which is the main result of the paper.
Since, we have just proven that there exists a surjective homomorphism from $\A$ to $\Z$, it remains to show that $\dim(\A) \leq \dim(\Z)$ to prove that this homomorphism is an isomorphism. 
Our strategy will be to provide a generating family for $\A$ with cardinality equal to the dimension of $\Z$. This generating family will hence be a basis for the algebra $\A$.

\subsection{Paths\label{sec:path}} 

We start by introducing a formalism that will be useful in the following.

Let $N$ be a non-negative integer. A path $\gamma$ of length $N$ is a sequence of integers $(\gamma_i)_{i=0}^N$ with $\gamma_{i+1} \in \{\gamma_{i}-1,\gamma_i,\gamma_{i}+1\}$ for all $i=0,\dots,N-1$, and $0 \leq \gamma_i \leq 2s$ . The set of all possible paths $\gamma$ going from $a$ to $b$ with length $N$ will be denoted
\begin{equation}
	\Gamma^{(a,b)}_N := \{\gamma=(\gamma_i)_{i=0}^N \ | \ \gamma_0=a, \gamma_N=b \}.
\end{equation}
Note that $\Gamma^{(a,b)}_N$ is empty if $N< |a-b|$. The index $N$ will be dropped when referring to paths of all possible lengths, \textit{i.e.}
\begin{equation}
	\Gamma^{(a,b)} := \bigcup_{N\geq |a-b|} \Gamma^{(a,b)}_N.
\end{equation}

We will illustrate a path $\gamma=(\gamma_i)_{i=0}^N \in \Gamma^{(a,b)}_N$ by a diagram of $N+1$ connected dots, with the horizontal direction representing the values of the indices $i$ and the vertical direction representing the values $\gamma_i$. The boundary values $a$ and $b$ will be explicitly indicated at the beginning and at the end of the path. An example is given in Figure \ref{fig:egpath}.
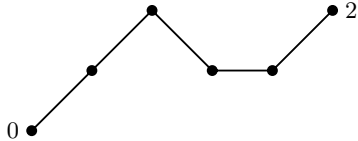
\begin{figure}[h]
	\begin{tikzpicture}[scale=1,line width=\lw]
		\node at (-0.25cm,0) {\footnotesize $0$};
		\draw[fill=black] (0,0) \bul -- ++\upath -- ++\upath -- ++\dpath -- ++\spath -- ++\upath;
		\node at (5*\xyd+0.25cm,2*\xyd) {\footnotesize $2$};
	\end{tikzpicture}
	\centering
	\caption{Diagrammatic representation of the path $\gamma=(0,1,2,1,1,2) \in \Gamma^{(0,2)}_5$.}
	\label{fig:egpath}
\end{figure}

Note that there is only one possible path of length $N=|a-b|$ going from $a$ to $b$: it is the strictly increasing path if $a<b$, the strictly decreasing path if $a>b$, or the zero length path if $a=b$. We will specifically denote it by $\overline{\gamma}(a,b)$. In the case where $a=b$, the path has zero length and it will be illustrated by a labeled dot $\begin{tikzpicture}[scale=1,line width=\lw,baseline={([yshift=-0.1cm]current bounding box.center)}]
	\draw[fill=black] (0,0) \bul;
	\node at (0.25cm,0) {\footnotesize $a$};
\end{tikzpicture}$.

To each path $\gamma$, we associate the following element of $\AWB$ (or of $\A$) 
\begin{equation}
	\x_\gamma= e_1^{(\gamma_0)} g_2 e_1^{(\gamma_1)} g_2 e_1^{(\gamma_2)} \dots e_1^{(\gamma_{N-1})} g_2 e_1^{(\gamma_N)}. \label{eq:Egam}
\end{equation}
If $\gamma = (\gamma_i)_{i=0}^M \in \Gamma_M^{(a,b)}$ and $\gamma'= (\gamma'_i)_{i=0}^N \in \Gamma_N^{(b,c)}$ are two paths, then their horizontal concatenation is defined by
\begin{equation}
	\gamma \cdot \gamma' = (\gamma_0,\gamma_1,\dots,\gamma_{M-1},\gamma'_0,\gamma'_1,\dots,\gamma'_N) \in \Gamma_{M+N}^{(a,c)} .
\end{equation}
In this case, we can write $\x_\gamma \x_{\gamma'} = \x_{\gamma \cdot \gamma'}$.
If $\gamma$ and $\gamma'$ are instead two paths with $\gamma_M \neq \gamma'_0$, then the orthogonality property of the idempotents $e_1^{(r)}$ implies that $\x_\gamma \x_{\gamma'} = 0$.

\subsection{A basis for $\A$}\label{sec:bas}  

From the PBW basis of the Askey--Wilson algebra (see \emph{e.g.}\ \cite{CFGPRV}), it is seen that a generating family for $\A$ is given by
\begin{equation}
	\{ g_1^a g_2^b (\sigma_2^{-1} g_1 \sigma_2)^c \kappa^p\ | \ a,b,c,p\in \mathbb{N} \}. \label{eq:genfam1}
\end{equation}
Due to the characteristic polynomials \eqref{eq:chareqg} for $g_1$ and $g_2$, this generating family can be reduced to
\begin{equation}
	\{ g_1^a g_2^b \sigma_2^{-1} g_1^c \sigma_2\kappa^p\ | \ 0\leq a,b,c \leq 2\s, \ p\in \mathbb{N} \}. \label{eq:genfam2} 
\end{equation}
The elements $g_1^a$, $g_2^b\sigma_2^{-1}$ and $g_1^c$ can be written respectively in terms of linear combinations of $e_1^{(a)}$, $g_2^b$ and $e_1^{(c)}$.  
Using moreover the fact that $\sigma_2$ is invertible and $\kappa$ central, one deduces that the set 
\begin{equation}
	\{ e_1^{(a)} g_2^b e_1^{(c)} \kappa^p\ | \ 0\leq a,b,c \leq 2\s, \ p\in \mathbb{N} \} \label{eq:genfam3}
\end{equation}
is also a generating family. To reduce this set further, we need some additional properties of the quotient $\A$ of the Askey--Wilson braid algebra.  

\begin{prop}\label{prop:AWB1}
	For any integers $0\leq a,c \leq 2s$ and $b\geq 0$,
	\begin{equation}	
		e_1^{(a)}  g_2^b e_1^{(c)} = 0 \quad \text{if } |a-c|>b.  \label{eq:re1}
	\end{equation}
\end{prop}
\begin{proof}
	The case $b=0$ is obvious from the second relation in \eqref{eq:epropr}. To obtain the case $b=1$,
	multiply the first Askey--Wilson relation in \eqref{AW-triple} by $e_1^{(a)}$ on the left and by $e_1^{(c)}$ on the right with $a\neq c$. Using the properties of the idempotents and the fact that $\kappa$ is central, one finds
	\begin{equation}
		\left( \frac{q\chi_a-\qi\chi_c}{q^2-q^{-2}}+q_aq_c^{-1} \right)e_1^{(a)}  g_2 e_1^{(c)} =0. \label{eq:aAW1c} 
	\end{equation}
	For $q$ generic, 
	the multiplicative factor on the LHS of \eqref{eq:aAW1c} does not vanish for $|a-c|>1$, which implies  relation \eqref{eq:re1} with $b=1$. The other cases are then proved by induction on $b$. Indeed, suppose that \eqref{eq:re1} holds for some integer $0 \leq b \leq 2s$. One can write
	\begin{equation}
		e_1^{(a)}  g_2^{b+1} e_1^{(c)} 
		= \sum_{r=0}^{2\s} e_1^{(a)} g_2 e_1^{(r)} g_2^{b} e_1^{(c)} 
		= e_1^{(a)} g_2 e_1^{(a-1)} g_2^{b} e_1^{(c)} + e_1^{(a)} g_2 e_1^{(a)} g_2^{b} e_1^{(c)}+e_1^{(a)} g_2 e_1^{(a+1)} g_2^{b} e_1^{(c)}.
	\end{equation} 
	By hypothesis, each term on the RHS of the previous equation is seen to be zero if $|a-c|>b+1$, which concludes the proof.  
\end{proof}
\begin{rem}
The preceding proposition simply expresses that $g_2$ acts tridiagonally in the eigenbasis of $g_1$. This is natural since it is well-known that the Askey--Wilson relations are closely related to Leonard systems {\normalfont \cite{Ter}}.
\end{rem}
It is seen from Proposition \ref{prop:AWB1} that the generating family \eqref{eq:genfam3} can be reduced to 
\begin{equation}
	\{ e_1^{(a)} g_2^b e_1^{(c)} \kappa^p\ | \ 0\leq a,c \leq 2\s, \ |a-c| \leq b \leq 2\s, \ p\in \mathbb{N} \}. \label{eq:genfam4}
\end{equation}
By inserting the identity between each factor of $g_2$ in the form of a sum of the idempotents $e_1^{(r)}$, as done in the proof of Proposition \ref{prop:AWB1}, it is also seen that
\begin{equation}
	e_1^{(a)} g_2^{b} e_1^{(c)} = \sum_{\gamma \in \Gamma^{(a,c)}_{b}} \x_\gamma = \sum_{\gamma \in \Gamma^{(a,c)}_{b}} e_1^{(a)} g_2 e_1^{(\gamma_1)} g_2 e_1^{(\gamma_2)} \dots e_1^{(\gamma_{b-1})} g_2 e_1^{(c)}, \label{eq:acpath}
\end{equation}
using the formalism of paths introduced in Section \ref{sec:path}. In particular, when $b=|a-c|$, equation \eqref{eq:acpath} implies that
\begin{equation}
	e_1^{(a)} g_2^{|a-c|} e_1^{(c)} = \x_{\overline{\gamma}(a,c)} , \label{eq:egemin}
\end{equation}
where we recall that $\overline{\gamma}(a,c)$ is the unique path in $\Gamma^{(a,c)}_{|a-c|}$. The following proposition allows to express the elements $\x_\gamma$ associated to the paths $\gamma$ in $\Gamma^{(a,c)}_{b}$ with $b>|a-c|$ in terms of \eqref{eq:egemin}.

\begin{prop}\label{prop:AWB2} For any integer $0\leq a \leq 2\s$,
	\begin{align}	
		&e_1^{(a)} g_2 e_1^{(a)} =  \tau_a(\kappa)  e_1^{(a)}, \label{eq:re2} \\
		&e_1^{(a)} g_2 e_1^{(a-1)}g_2 e_1^{(a)} 
		= F_a(\kappa) e_1^{(a)}, \label{eq:re3} \\
		&e_1^{(a)} g_2 e_1^{(a+1)}g_2 e_1^{(a)} 
		= F_{a+1}(\kappa) e_1^{(a)}, \label{eq:re4}
	\end{align}
	with the convention $e_1^{(-1)}=e_1^{(2\s+1)}=0$ and with
	\begin{equation}
		\tau_a(\kappa)= \frac{\chi_\s(\kappa+\chi_\s)}{\chi_0+\chi_{a}}, \quad F_a(\kappa) = -\frac{[2\s+1-a]_q[2\s+1+a]_q}{(q^a+q^{-a})^2[2a-1]_q[2a+1]_q}(\kappa-\chi_{\s-a})(\kappa-\chi_{\s+a}), \label{eq:tauF}
	\end{equation}
where $\displaystyle [n]_q=\frac{q^n-q^{-n}}{q-q^{-1}}$.
\end{prop}
\begin{proof}
	Equation \eqref{eq:re2} is obtained by multiplying the first Askey--Wilson relation in \eqref{AW-triple} by $e_1^{(a)}$ both on the left and on the right, and by using the properties of the idempotents. Proceeding similarly with the third Askey--Wilson relation in \eqref{AW-triple}, one gets
	\begin{equation}
		q\ q_a^{-1} e_1^{(a)} g_2 \sigma_1 g_2 e_1^{(a)} - \qi q_a e_1^{(a)} g_2 \sigma_1^{-1} g_2 e_1^{(a)} = (q-\qi)(\chi_\s(\kappa+\chi_s) - \chi_0\chi_a)e_1^{(a)}. \label{eq:aAW3a}
	\end{equation}
	The element $\sigma_1^{\pm1}$ in the previous equation can be written as a sum of the idempotents $e_1^{(r)}$ with coefficients $q_r^{\pm1}$, as in \eqref{eq:sige}. Because of \eqref{eq:re1}, only the terms with $r=a-1,a,a+1$ in this sum remain in \eqref{eq:aAW3a}. Using \eqref{eq:re2} and simplifying the result, one thus finds  
	\begin{equation}	
		[2a+3]_q e_1^{(a)} g_2 e_1^{(a+1)}g_2 e_1^{(a)} -[2a-1]_q e_1^{(a)} g_2 e_1^{(a-1)}g_2 e_1^{(a)} 
		= (\tau_a(\kappa)-\chi_{0})(\tau_a(\kappa)-\chi_{a}) e_1^{(a)}. \label{eq:aAW3a2}
	\end{equation}
	One can also multiply equation \eqref{eq:Om} which defines the central element $\Omega$ on both sides by $e_1^{(a)}$ to obtain similarly
	\begin{align}
		&q^{2a+1}[2a+3]_q e_1^{(a)} g_2 e_1^{(a+1)}g_2 e_1^{(a)} -q^{-2a-1}[2a-1]_q e_1^{(a)} g_2 e_1^{(a-1)}g_2 e_1^{(a)} \label{eq:aOma} \\
		= &\frac{1}{q-q^{-1}} \left(\Omega +q^2\chi_a^2 - \chi_\s(\kappa+\chi_\s)(q\chi_a+\chi_0\tau_a(\kappa))  +(q^2+q^{-2}+q\chi_a)\tau_a^2(\kappa)\right)e_1^{(a)}. \nonumber
	\end{align} 
	Equations \eqref{eq:aAW3a2} and \eqref{eq:aOma} can be solved for $e_1^{(a)} g_2 e_1^{(a\pm1)}g_2 e_1^{(a)}$. With the help of the defining relation \eqref{eq:AWB1}, these solutions are given by \eqref{eq:re3} and \eqref{eq:re4}. 
\end{proof}

The results of Proposition \ref{prop:AWB2} can be understood diagrammatically in terms of paths using
\begin{align}
	\begin{tikzpicture}[scale=1,line width=\lw,baseline={([yshift=-0.1cm]current bounding box.center)}]
		\node at (-0.25cm,0) {\footnotesize $a$};
		\draw[fill=black] (0,0) \bul -- ++\spath;
		\node at (\xyd+0.25cm,0*\xyd) {\footnotesize $a$};
	\end{tikzpicture} \ &= \ \tau_a(\kappa) \
	\begin{tikzpicture}[scale=1,line width=\lw,baseline={([yshift=-0.1cm]current bounding box.center)}]
		\draw[fill=black] (0,0) \bul;
		\node at (0.25cm,0*\xyd) {\footnotesize $a$};
	\end{tikzpicture}, \label{eq:trait} \\
	\begin{tikzpicture}[scale=1,line width=\lw,baseline={([yshift=-0.1cm]current bounding box.center)}]
		\node at (-0.25cm,0) {\footnotesize $a$};
		\draw[fill=black] (0,0) \bul -- ++\dpath -- ++\upath;
		\node at (2*\xyd+0.25cm,0*\xyd) {\footnotesize $a$};
	\end{tikzpicture} \ &= \ F_{a}(\kappa) \
	\begin{tikzpicture}[scale=1,line width=\lw,baseline={([yshift=-0.1cm]current bounding box.center)}]
		\draw[fill=black] (0,0) \bul;
		\node at (0.25cm,0*\xyd) {\footnotesize $a$};
	\end{tikzpicture}, \label{eq:creux} \\
	\begin{tikzpicture}[scale=1,line width=\lw,baseline={([yshift=-0.1cm]current bounding box.center)}]
		\node at (-0.25cm,0) {\footnotesize $a$};
		\draw[fill=black] (0,0) \bul -- ++\upath -- ++\dpath;
		\node at (2*\xyd+0.25cm,0*\xyd) {\footnotesize $a$};
	\end{tikzpicture} \ &= \ F_{a+1}(\kappa) \
	\begin{tikzpicture}[scale=1,line width=\lw,baseline={([yshift=-0.1cm]current bounding box.center)}]
		\draw[fill=black] (0,0) \bul;
		\node at (0.25cm,0*\xyd) {\footnotesize $a$};
	\end{tikzpicture}. \label{eq:pic}
\end{align}
Any path $\gamma \in \Gamma^{(a,c)}_b$ with $b>|a-c|$ is represented by a diagram that necessarily contains parts that look like the LHS of \eqref{eq:trait}--\eqref{eq:pic}. Equations \eqref{eq:trait}--\eqref{eq:pic} hence imply that the elements $\x_\gamma$ associated to these paths $\gamma \in \Gamma^{(a,c)}_b$ are equal to a factor depending on $\kappa$ multiplied by the element \eqref{eq:egemin} associated to the unique path in $\Gamma^{(a,c)}_{|a-c|}$. It follows from this discussion and from equation \eqref{eq:acpath} that for all integers $0 \leq a,c \leq 2\s$ and $n\geq 0$ there is a polynomial $P^{(a,c)}_{n}(\kappa)$ of degree at most $n$ in $\kappa$ such that
	\begin{equation}
		e_1^{(a)} g_2^{|a-c|+n} e_1^{(c)} = P_n^{(a,c)}(\kappa) e_1^{(a)} g_2^{|a-c|} e_1^{(c)}. \label{eq:minpath}
	\end{equation}
Therefore, the generating family \eqref{eq:genfam4} can be reduced to 
\begin{equation}
	\{ e_1^{(a)} g_2^{|a-c|} e_1^{(c)} \kappa^p\ | \ 0\leq a,c \leq 2\s, \ p\in \mathbb{N} \}. \label{eq:genfam5}
\end{equation}
Our goal will now be to further reduce the previous set by restraining the possible values of $p$.

\begin{prop} \label{prop:charkape} For any integer $0 \leq a \leq 2\s$,
	\begin{equation}
		\prod_{r=|a-\s|}^{a+\s}(\kappa-\chi_{r})e_1^{(a)} = 0. \label{eq:polkappae}
	\end{equation}	
	In particular,
	\begin{equation}
		(\kappa-\chi_\s)e_1^{(0)}=0. \label{eq:kap0}
	\end{equation}
\end{prop}
\begin{proof}
	The case $a=0$, which is \eqref{eq:kap0}, can be obtained by comparing the defining relation \eqref{eq:AWB2} of $\A$ with the case $a=0$ of equation \eqref{eq:re2}. Then, for $0< a \leq \s$, one can use relations \eqref{eq:re3} and \eqref{eq:kap0} to proceed as follows:
	\begin{align}
		(\kappa-\chi_\s)F_{1}F_{2} \dots F_{a-1}F_{a}e_1^{(a)} 
		&= (\kappa-\chi_\s)F_{1}F_{2} \dots F_{a-1}e_1^{(a)} g_2 e_1^{(a-1)}g_2 e_1^{(a)} \label{eq:F1Fa}\\
		&=(\kappa-\chi_\s)F_{1}F_{2} \dots F_{a-2}e_1^{(a)} g_2 e_1^{(a-1)} g_2 e_1^{(a-2)}g_2 e_1^{(a-1)}g_2 e_1^{(a)}\\
		&=(\kappa-\chi_\s)e_1^{(a)} g_2 e_1^{(a-1)} \dots e_1^{(1)}g_2 e_1^{(0)}g_2e_1^{(1)} \dots e_1^{(a-1)}g_2 e_1^{(a)}\\
		&=0.
	\end{align}
	In the previous equations, the dependence on $\kappa$ of the functions $F_i$ is not written for simplicity. From the explicit expression of these functions given in \eqref{eq:tauF}, it is seen that the LHS of \eqref{eq:F1Fa} contains only once all the factors $(\kappa-\chi_r)$ with $r=\s-a,\s-a+1,\dots,\s-1,\s,\s+1,\dots,\s+a-1,\s+a$ and nothing more except some non-vanishing constants that can be simplified from the equation. This proves \eqref{eq:polkappae} also in this case.
	 
	For the case $\s<a\leq 2\s$, the previous method gives more factors $(\kappa-\chi_r)$ than desired. One can instead multiply the second relation in \eqref{rel-si-g} by $e_1^{(a)}$ on both sides and use equation \eqref{eq:sige} together with \eqref{eq:re1} to obtain
	\begin{equation}
		\chi_a q_a e_1^{(a)} \sigma_2 e_1^{(a)}-\sum_{r=a,a\pm 1} q_r e_1^{(a)} g_2 e_1^{(r)} \sigma_2 e_1^{(a)}=0. \label{eq:abraida}
	\end{equation}	
	Let us recall that $\sigma_2$ can be written as a linear combination of the elements $g_2^n$ for $n=0,1,\dots,2\s$, as seen from equations \eqref{eq:sige} and \eqref{eq:eprodg}. Proposition \ref{prop:AWB2} and equation \eqref{eq:minpath} then imply that the LHS of \eqref{eq:abraida} is equal to $Q^{(a)}(\kappa) e_1^{(a)}$, where $Q^{(a)}(\kappa)$ is a polynomial in $\kappa$ of maximal degree $2\s+1$. It can be shown that this polynomial does not vanish identically. This is the most difficult part of the proof and the details are given in Appendix \ref{app}. Besides, we know that the image of $\kappa$ in the centralizer $\Z$ satisfies the polynomial equation \eqref{eq:polkappae} which is of degree $2\s+1$ exactly. Since this polynomial is minimal in the centralizer algebra and since the map from $\A$ to $\Z$ is surjective, we conclude that $Q^{(a)}(\kappa)$ is equal to the polynomial of \eqref{eq:polkappae} up to a non-zero multiplicative factor. 
\end{proof}

\begin{rem}
Considering the least common multiple of the polynomials which appear in \eqref{eq:polkappae} for $a=0,1,\dots,2\s$, we find as a consequence of Proposition \refeq{prop:charkape} that
\begin{equation}
		\prod_{r= j_\text{min}}^{3s}  (\kappa -\chi_{r} )=0, \label{eq:chareqkap}
	\end{equation}
where $j_\text{min}$ is given in \eqref{eq:jmin}. Note that \eqref{eq:chareqkap} also follows immediately from the isomorphism with the centralizer, see \eqref{eq:repCasimir3}.
\end{rem}

One deduces from Proposition \ref{prop:charkape} that
\begin{equation}
	Q^{(a,c)}(\kappa) \ e_1^{(a)} g_2^{|a-c|} e_1^{(c)} = 0,
\end{equation}
where $Q^{(a,c)}(\kappa)$ is the greatest common divisor of the annihilating polynomials of $\kappa$ on $e_1^{(a)}$ and $e_1^{(c)}$ and has degree $n(a,c)+1$ with
\begin{equation}
	n(a,c)= \begin{cases}
		2\min\{a,c\}, &\text{if } a+c \leq 2s, \\
		2s-|a-c|, &\text{if } a+c > 2s.
	\end{cases} \label{eq:ndegree}
\end{equation}  
It follows that the set
\begin{equation}
	\{ e_1^{(a)} g_2^{|a-c|} e_1^{(c)} \kappa^p\ | \ 0\leq a,c \leq 2\s, \ 0\leq p \leq n(a,c) \} \label{eq:genfam6}
\end{equation}
is a generating family for the algebra $\A$. Comparing the degrees \eqref{eq:ndegree} with the dimensions \eqref{eq:dimcentr2}, it is seen that the cardinality of the set \eqref{eq:genfam6} is equal to the dimension of the centralizer $\Z$. We can therefore conclude that $\A$ is isomorphic to $\Z$ and that the set \eqref{eq:genfam6} is a basis.

\begin{rem}\label{rem:relkape0}
	In the quotient of $\AWB$ by the special Askey--Wilson relation \eqref{eq:AWB1b} only, it is possible to show that
	\begin{equation}
		(\kappa-\chi_{\s})^2e_1^{(0)}=0. \label{eq:re5}
	\end{equation}
	Indeed, \eqref{eq:re5} follows from equation \eqref{eq:re3} with $a=0$. We were not able however to show that \eqref{eq:kap0} holds in this quotient for any spin $s$, which explains the addition of relation \eqref{eq:AWB2b} in Theorem \refeq{thm:iso}. Note that for specific values of the spin $s$, we have a procedure for proving that \eqref{eq:kap0} is implied by the other relations which goes as follows. First, one can multiply the characteristic polynomial equation \eqref{eq:chareqg} of $g_2$ by the idempotent $e_1^{(0)}$ on the left and on the right. The LHS of the resulting equation can be expressed as a sum over elements $\x_\gamma$ where $\gamma \in \Gamma_{2s+1}^{(0,0)}$. One can then use the relations of Proposition \refeq{prop:AWB2} to write this sum of elements as a polynomial in $\kappa$ of maximal degree $2s+1$ multiplied by $e_1^{(0)}$. For $s=1/2,1,3/2,2$ we have verified that this polynomial contains the factor $(\kappa-\chi_s)$ only once. In these cases, one then deduces from \eqref{eq:re5} that \eqref{eq:kap0} holds.      
\end{rem}

\subsection{A basis using paths}

We will now provide another basis for the centralizer $\Z$ which makes use of the path elements $\x_\gamma$ defined in \eqref{eq:Egam}.

Let $a$ be an integer such that $0\leq a \leq 2s$. We define a sequence $L_a$ of paths in $\Gamma^{(a,a)}$ with lengths that increase by one at each step as follows:
\begin{equation}
	L_a = \left((a), (a,a), (a,a-1,a), (a,a-1,a-1,a), \dots, (a,a-1,\dots,1,0,1,\dots,a-1,a)\right).
\end{equation}
Note that the sequence $L_a$ stops after the unique path $\gamma=(\gamma_i)_{i=0}^{N} \in \Gamma^{(a,a)}$ of length $N=2a$ that contains only once the value zero at the position $i=a$. The cardinality of $L_a$ is thus 
\begin{equation}
	|L_a| = 2a+1. \label{eq:cardL}
\end{equation}
A diagrammatic illustration of the elements of the sequences $L_a$ for $a=0,1,2$ is given in Figure \ref{fig:Lpaths}.
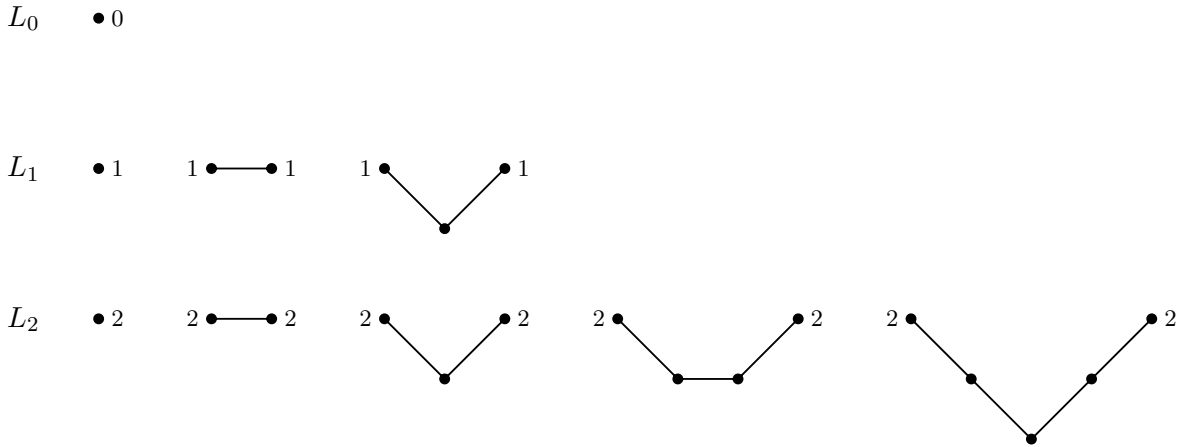
\begin{figure}[h]
	\begin{tikzpicture}[scale=1,line width=\lw,baseline={([yshift=-0.1cm]current bounding box.center)}]
		\node at (-1cm,0) {$L_0$}; 
		\draw[fill=black] (0,0) \bul;
		\node at (0.25cm,0*\xyd) {\footnotesize $0$};
		
		\node at (-1cm,-2cm) {$L_1$};
		\begin{scope}[yshift=-2cm]
			\draw[fill=black] (0,0) \bul;
			\node at (0.25cm,0*\xyd) {\footnotesize $1$};
			
			\begin{scope}[xshift=1.5cm]
				\node at (-0.25cm,0) {\footnotesize $1$};
				\draw[fill=black] (0,0) \bul -- ++\spath;
				\node at (1*\xyd+0.25cm,0*\xyd) {\footnotesize $1$};
			\end{scope}
			\begin{scope}[xshift=2*1.5cm+\xyd]
				\node at (-0.25cm,0) {\footnotesize $1$};
				\draw[fill=black] (0,0) \bul -- ++\dpath -- ++\upath;
				\node at (2*\xyd+0.25cm,0*\xyd) {\footnotesize $1$};
			\end{scope}
		\end{scope}
	
		\node at (-1cm,-4cm) {$L_2$};
		\begin{scope}[yshift=-4cm]
			\draw[fill=black] (0,0) \bul;
			\node at (0.25cm,0*\xyd) {\footnotesize $2$};
			
			\begin{scope}[xshift=1.5cm]
				\node at (-0.25cm,0) {\footnotesize $2$};
				\draw[fill=black] (0,0) \bul -- ++\spath;
				\node at (1*\xyd+0.25cm,0*\xyd) {\footnotesize $2$};
			\end{scope}
			\begin{scope}[xshift=2*1.5cm+\xyd]
				\node at (-0.25cm,0) {\footnotesize $2$};
				\draw[fill=black] (0,0) \bul -- ++\dpath -- ++\upath;
				\node at (2*\xyd+0.25cm,0*\xyd) {\footnotesize $2$};
			\end{scope}
		
			\begin{scope}[xshift=3*1.5cm+3*\xyd]
				\node at (-0.25cm,0) {\footnotesize $2$};
				\draw[fill=black] (0,0) \bul -- ++\dpath -- ++\spath -- ++\upath;
				\node at (3*\xyd+0.25cm,0*\xyd) {\footnotesize $2$};
			\end{scope}
			
			\begin{scope}[xshift=4*1.5cm+6*\xyd]
				\node at (-0.25cm,0) {\footnotesize $2$};
				\draw[fill=black] (0,0) \bul -- ++\dpath -- ++\dpath -- ++\upath -- ++\upath;
				\node at (4*\xyd+0.25cm,0*\xyd) {\footnotesize $2$};
			\end{scope}
		\end{scope}
	\end{tikzpicture}
	\centering
	\caption{Elements of the sequences $L_a$ for $a=0,1,2$.}
	\label{fig:Lpaths}
\end{figure}

Consider the following set of paths:
\begin{equation}
	\mathcal{P}_s^{(a,c)} = 
	\begin{cases}
		\{ \gamma \cdot \overline{\gamma}{(a,c)} \ | \ \gamma=(\gamma_i)_{i=0}^N \in L_a, \  N \leq 2s-(c-a)  \}, & \text{if } a \leq c, \\
		\{ \overline{\gamma}{(a,c)} \cdot \gamma  \ | \ \gamma=(\gamma_i)_{i=0}^N \in L_c, \ N \leq 2s-(a-c)  \}, & \text{if } a > c.
	\end{cases}	  \label{eq:Psac} 
\end{equation} 
The set $\mathcal{P}_s^{(a,c)}$ consists of concatenating the paths $\gamma \in L_{\min\{a,c\}}$ of length bounded by $2s-|a-c|$ on the lowest endpoint of the direct path $\overline{\gamma}{(a,c)}$ going from $a$ to $c$ with length $|a-c|$. The result of the concatenation is a path going from $a$ to $c$ of maximal length $2s$. As an example, Figure \ref{fig:basispaths} illustrates all the paths in the set $\mathcal{P}_2^{(2,4)}$, which can be obtained from the sequence $L_2$ in Figure \ref{fig:Lpaths}. Recall that the longest path in $L_{\min\{a,c\}}$ has length $2\min\{a,c\}$. If $2\min\{a,c\} \leq 2s-|a-c|$, that is if $a+c \leq 2s$, then it follows from \eqref{eq:cardL} that $|\mathcal{P}_s^{(a,c)}| = |L_{\min\{a,c\}}|= 2\min\{a,c\}+1$. If $2\min\{a,c\} > 2s-|a-c|$, that is if $a+c > 2s$, then $|\mathcal{P}_s^{(a,c)}|=2s-|a-c|+1$. In summary,
\begin{equation}
	|\mathcal{P}_s^{(a,c)}| =
	\begin{cases}
		2\min\{a,c\}+1, & \text{if }  a+c \leq 2s,\\
		2s-|a-c|+1, & \text{if } a+c > 2s.
	\end{cases} \label{eq:cardP}
\end{equation}  
\begin{figure}[h]
	\begin{tikzpicture}[scale=1,line width=\lw,baseline={([yshift=-0.1cm]current bounding box.center)}]
		
		\node at (-1.5cm,0) {$\mathcal{P}_2^{(2,4)}$};
			\node at (-0.25cm,0) {\footnotesize $2$};
			\draw[fill=black] (0,0) \bul -- ++\upath -- ++\upath;
			\node at (2*\xyd+0.25cm,2*\xyd) {\footnotesize $4$};
			
			\begin{scope}[xshift=1.5cm+2*\xyd]
				\node at (-0.25cm,0) {\footnotesize $2$};
				\draw[fill=black] (0,0) \bul -- ++\spath -- ++\upath -- ++\upath;
				\node at (3*\xyd+0.25cm,2*\xyd) {\footnotesize $4$};
			\end{scope}
			\begin{scope}[xshift=2*1.5cm+5*\xyd]
				\node at (-0.25cm,0) {\footnotesize $2$};
				\draw[fill=black] (0,0) \bul -- ++\dpath -- ++\upath -- ++\upath -- ++\upath;
				\node at (4*\xyd+0.25cm,2*\xyd) {\footnotesize $4$};
			\end{scope}
	\end{tikzpicture}
	\centering
	\caption{Elements of the set $\mathcal{P}_2^{(2,4)}$.}
	\label{fig:basispaths}
\end{figure}
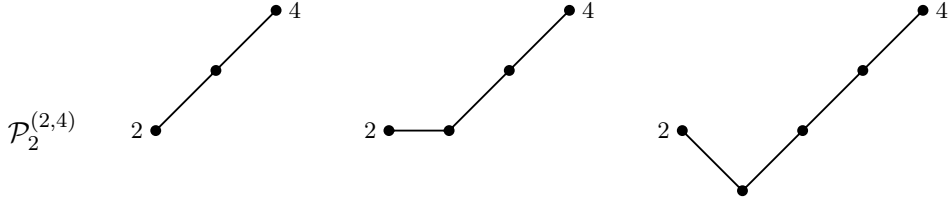

Consider now the following set of elements in $\AWB$ (with $\x_\gamma$ defined in \eqref{eq:Egam})
\begin{equation}
	\bigcup_{a,c=0}^{2s}
	\{ x_\gamma \ | \ \gamma \in \mathcal{P}_s^{(a,c)} \}. \label{eq:basispaths}
\end{equation}
Comparing \eqref{eq:cardP} with \eqref{eq:dimcentr2} and \eqref{eq:ndegree}, it is seen that the set \eqref{eq:basispaths} has cardinality equal to the dimension of the algebra $\A \cong \Z$. Therefore, in order to show that it is a basis for $\A$, it is sufficient to verify that it is a generating family. From the definition \eqref{eq:Psac}, any element in \eqref{eq:basispaths} can be written as
\begin{equation}
	\x_{\gamma \cdot \overline{\gamma}{(a,c)}} = \x_{\gamma} \x_{\overline{\gamma}{(a,c)}} \quad \text{or} \quad \x_{\overline{\gamma}{(a,c)} \cdot \gamma} = \x_{\overline{\gamma}{(a,c)}} \x_{\gamma}, \label{eq:basispathselem}
\end{equation}
where $\gamma$ is a path in $L_{\min\{a,c\}}$ with maximal length $|\mathcal{P}_s^{(a,c)}|-1=n(a,c)$, see equations \eqref{eq:ndegree} and \eqref{eq:cardP}.
Using relations \eqref{eq:trait} and \eqref{eq:creux}, it is observed that any element $\x_{\gamma}$ in \eqref{eq:basispathselem} can be expressed as a polynomial in $\kappa$ of degree equal to the length of $\gamma$. It then follows from \eqref{eq:egemin} and \eqref{eq:basispathselem} that any element of the basis \eqref{eq:genfam6} can be given in terms of the elements of the set \eqref{eq:basispaths} with a triangular transformation. We conclude that \eqref{eq:basispaths} is a basis for $\A \cong \Z$.  

\begin{rem}
	This description of the centralizer $\Z$ in terms of a basis given as paths allows for a diagrammatic interpretation of the relation \eqref{eq:AWB2b} added in the quotient of $\AWB$. Indeed, this relation means that the element $e_1^{(0)}g_2e_1^{(0)}$ associated to the diagram
	$\begin{tikzpicture}[scale=1,line width=\lw,baseline={([yshift=-0.1cm]current bounding box.center)}]
		\node at (-0.25cm,0) {\footnotesize $0$};
		\draw[fill=black] (0,0) \bul -- ++\spath;
		\node at (1*\xyd+0.25cm,0*\xyd) {\footnotesize $0$};
	\end{tikzpicture}$ is proportional to the element $e_1^{(0)}$ associated to the diagram $\begin{tikzpicture}[scale=1,line width=\lw,baseline={([yshift=-0.1cm]current bounding box.center)}]
	\draw[fill=black] (0,0) \bul;
	\node at (0.25cm,0*\xyd) {\footnotesize $0$};
\end{tikzpicture}$. Therefore relation \eqref{eq:AWB2b} forbids the use of paths containing a horizontal segment placed on the vertical level $0$ in the construction of the basis \eqref{eq:basispaths}. 
\end{rem}  

\subsection{Temperley--Lieb and Birman--Murakami--Wenzl algebras}
We conclude this section by briefly considering two explicit examples of the algebra $\A \cong \Z$, namely when $s=1/2$ and $s=1$. In these cases, the centralizer algebra $\Z$ is known to be isomorphic to the Temperley--Lieb and Birman--Murakami--Wenzl algebras respectively (on three strands). We hence provide the connection between the known presentation of the centralizer in terms of these algebras and the presentation using the Askey--Wilson braid algebra obtained in this paper.   

Let us start with the case $s=1/2$. The dimension of the centralizer $Z_{1/2}\cong \mathcal{A}_{1/2}$ is $5$, as can be computed directly from \eqref{eq:dimcentr}. For each $i=1,2$, there are only two idempotent elements $e_i^{(0)}$ and $e_i^{(1)}$, and the characteristic equations \eqref{eq:chareqsig} and \eqref{eq:chareqg} for $\sigma_i$ and $g_i$ are of degree $2$. The elements of the basis \eqref{eq:genfam6} in this case are
\begin{equation}
	\{e_1^{(0)},e_1^{(0)}g_2e_1^{(1)},e_1^{(1)}g_2e_1^{(0)},e_1^{(1)},e_1^{(1)}\kappa\}.
\end{equation}
Moreover, the paths of the set \eqref{eq:Psac} that are used to construct the basis elements \eqref{eq:basispaths} are all illustrated in Figure \ref{fig:basispathsTL}.
\begin{figure}[h]
	\begin{tikzpicture}[scale=1,line width=\lw,baseline={([yshift=-0.1cm]current bounding box.center)}]
		
		\draw[fill=black] (0,0) \bul;
		\node at (0.25cm,0*\xyd) {\footnotesize $0$};
		
		\begin{scope}[xshift=1.5cm]
			\node at (-0.25cm,0) {\footnotesize $0$};
			\draw[fill=black] (0,0) \bul -- ++\upath;
			\node at (\xyd+0.25cm,1*\xyd) {\footnotesize $1$};
		\end{scope}
		
		\begin{scope}[xshift=2*1.5cm+\xyd]
			\node at (-0.25cm,\xyd) {\footnotesize $1$};
			\draw[fill=black] (0,\xyd) \bul -- ++\dpath;
			\node at (\xyd+0.25cm,0) {\footnotesize $0$};
		\end{scope}
		
		\begin{scope}[xshift=3*1.5cm+2*\xyd,yshift=\xyd]
			\draw[fill=black] (0,0) \bul;
			\node at (0.25cm,0*\xyd) {\footnotesize $1$};
			
			\begin{scope}[xshift=1.5cm]
				\node at (-0.25cm,0) {\footnotesize $1$};
				\draw[fill=black] (0,0) \bul -- ++\spath;
				\node at (1*\xyd+0.25cm,0*\xyd) {\footnotesize $1$};
			\end{scope}
		\end{scope}
	\end{tikzpicture}
	\centering
	\caption{Paths associated to the basis elements \eqref{eq:basispaths} in the case $s=1/2$.}
	\label{fig:basispathsTL}
\end{figure}
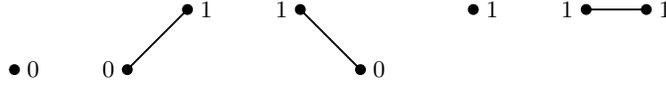

Because of the first relation in \eqref{eq:epropr}, one can write
\begin{equation}
	e_i^{(1)}=1-e_i^{(0)}.
\end{equation}
It then follows from \eqref{eq:sige} and \eqref{eq:gi} that for $i=1,2$
\begin{align}
	&\sigma_i = q_1 + (q_0-q_1)e_i^{(0)} = -q^{2} + q(q+\qi)e_i^{(0)}, \label{eq:sigTL} \\
	&g_i = \chi_1 + (\chi_0-\chi_1)e_i^{(0)} = (q^3+q^{-3}) - (q-\qi)^2(q+\qi)e_i^{(0)}. \label{eq:gTL} 
\end{align}
We now provide the connection with the Temperley--Lieb algebra.
\begin{defi}{\normalfont \cite{TL}} 
	The Temperley--Lieb algebra $TL_3(q)$ is generated by $e_1$ and $e_2$ with the following defining relations
	\begin{align}
		&e_1^2=(q+\qi)e_1, \quad e_2^2=(q+\qi)e_2, \label{eq:TL1} \\
		&e_1e_2e_1=e_1, \quad e_2e_1e_2=e_2. \label{eq:TL2} 
	\end{align}
\end{defi}
On one hand, we have shown that the quotient $\mathcal{A}_{1/2}$ is isomorphic to the centralizer $Z_{1/2}$, and on the other hand, it is known that $Z_{1/2}$ is isomorphic to $TL_3(q)$. One can then verify that the following mapping extends to an algebra isomorphism from $\mathcal{A}_{1/2}$ to $TL_3(q)$: 
\begin{equation}
	e_i^{(0)} \mapsto \frac{1}{(q+\qi)}e_i, \quad \text{for } i=1,2. \label{eq:mapTL}
\end{equation}
It follows from \eqref{eq:sigTL} and \eqref{eq:gTL} that for $i=1,2$
\begin{align}
	&\sigma_i \mapsto  -q^{2} + qe_i, \\
	&g_i \mapsto (q^3+q^{-3}) - (q-\qi)^2e_i. 
\end{align}

Let us now consider similarly the case $s=1$. The dimension of the centralizer $Z_1 \cong \mathcal{A}_{1} $ is $15$. For each $i=1,2$, there are three idempotents $e_i^{(0)},e_i^{(1)},e_i^{(2)}$ in the algebra $\mathcal{A}_{1}$ and the characteristic equations for $\sigma_i$ and $g_i$ are of degree 3. The basis elements \eqref{eq:genfam6} are 
\begin{align}
	\{e_1^{(0)},e_1^{(0)}g_2e_1^{(1)},e_1^{(0)}g_2^2e_1^{(2)},e_1^{(1)}g_2e_1^{(0)},e_1^{(1)} ,e_1^{(1)}\kappa,e_1^{(1)}\kappa^2,e_1^{(1)}g_2e_1^{(2)},e_1^{(1)}g_2e_1^{(2)}\kappa, \nonumber\\
	e_1^{(2)}g_2^2e_1^{(0)},e_1^{(2)}g_2e_1^{(1)},e_1^{(2)}g_2e_1^{(1)}\kappa,e_1^{(2)},e_1^{(2)}\kappa,e_1^{(2)}\kappa^2\},
\end{align}
and the paths associated to the basis \eqref{eq:basispaths} are illustrated in Figure \ref{fig:basispathsBMW}. In this case, there is a connection with the Birman--Murakami--Wenzl algebra.
\begin{figure}[h]
	\begin{tikzpicture}[scale=1,line width=\lw,baseline={([yshift=-0.1cm]current bounding box.center)}]
		
		\draw[fill=black] (0,0) \bul;
		\node at (0.25cm,0*\xyd) {\footnotesize $0$};
		
		\begin{scope}[xshift=1.5cm]
			\node at (-0.25cm,0) {\footnotesize $0$};
			\draw[fill=black] (0,0) \bul -- ++\upath;
			\node at (\xyd+0.25cm,1*\xyd) {\footnotesize $1$};
		\end{scope}
	
		\begin{scope}[xshift=2*1.5cm+\xyd]
			\node at (-0.25cm,0) {\footnotesize $0$};
			\draw[fill=black] (0,0) \bul -- ++\upath -- ++\upath;
			\node at (2*\xyd+0.25cm,2*\xyd) {\footnotesize $2$};
		\end{scope}
		
		\begin{scope}[xshift=3*1.5cm+3*\xyd]
			\node at (-0.25cm,\xyd) {\footnotesize $1$};
			\draw[fill=black] (0,\xyd) \bul -- ++\dpath;
			\node at (\xyd+0.25cm,0) {\footnotesize $0$};
		\end{scope}
	
		\begin{scope}[xshift=4*1.5cm+4*\xyd]
			\draw[fill=black] (0,\xyd) \bul;
			\node at (0.25cm,1*\xyd) {\footnotesize $1$};
		\end{scope}
		
		\begin{scope}[xshift=5*1.5cm+4*\xyd]
			\node at (-0.25cm,\xyd) {\footnotesize $1$};
			\draw[fill=black] (0,\xyd) \bul -- ++\spath;
			\node at (1*\xyd+0.25cm,1*\xyd) {\footnotesize $1$};
		\end{scope}	
	
		\begin{scope}[xshift=6*1.5cm+5*\xyd]
			\node at (-0.25cm,\xyd) {\footnotesize $1$};
			\draw[fill=black] (0,\xyd) \bul  -- ++\dpath -- ++\upath;
			\node at (2*\xyd+0.25cm,1*\xyd) {\footnotesize $1$};
		\end{scope}	
	
		\begin{scope}[xshift=0*1.5cm+0*\xyd,yshift=-3*\xyd]
			\node at (-0.25cm,\xyd) {\footnotesize $1$};
			\draw[fill=black] (0,\xyd) \bul  -- ++\upath;
			\node at (1*\xyd+0.25cm,2*\xyd) {\footnotesize $2$};
		\end{scope}
	
		\begin{scope}[xshift=1*1.5cm+1*\xyd,yshift=-3*\xyd]
			\node at (-0.25cm,\xyd) {\footnotesize $1$};
			\draw[fill=black] (0,\xyd) \bul -- ++\spath  -- ++\upath;
			\node at (2*\xyd+0.25cm,2*\xyd) {\footnotesize $2$};
		\end{scope}		
		
		\begin{scope}[xshift=2*1.5cm+3*\xyd,yshift=-3*\xyd]
			\node at (-0.25cm,2*\xyd) {\footnotesize $2$};
			\draw[fill=black] (0,2*\xyd) \bul -- ++\dpath -- ++\dpath;
			\node at (2*\xyd+0.25cm,0) {\footnotesize $0$};
		\end{scope}
	
		\begin{scope}[xshift=3*1.5cm+5*\xyd,yshift=-3*\xyd]
			\node at (-0.25cm,2*\xyd) {\footnotesize $2$};
			\draw[fill=black] (0,2*\xyd) \bul -- ++\dpath;
			\node at (1*\xyd+0.25cm,\xyd) {\footnotesize $1$};
		\end{scope}
	
		\begin{scope}[xshift=4*1.5cm+6*\xyd,yshift=-3*\xyd]
			\node at (-0.25cm,2*\xyd) {\footnotesize $2$};
			\draw[fill=black] (0,2*\xyd) \bul -- ++\dpath -- ++\spath;
			\node at (2*\xyd+0.25cm,\xyd) {\footnotesize $1$};
		\end{scope}
	
		\begin{scope}[xshift=0*1.5cm+0*\xyd,yshift=-6*\xyd]
			\draw[fill=black] (0,2*\xyd) \bul;
			\node at (0*\xyd+0.25cm,2*\xyd) {\footnotesize $2$};
		\end{scope}
		
		\begin{scope}[xshift=1*1.5cm+0*\xyd,yshift=-6*\xyd]
			\node at (-0.25cm,2*\xyd) {\footnotesize $2$};
			\draw[fill=black] (0,2*\xyd) \bul -- ++\spath;
			\node at (1*\xyd+0.25cm,2*\xyd) {\footnotesize $2$};
		\end{scope}
	
		\begin{scope}[xshift=2*1.5cm+1*\xyd,yshift=-6*\xyd]
			\node at (-0.25cm,2*\xyd) {\footnotesize $2$};
			\draw[fill=black] (0,2*\xyd) \bul -- ++\dpath -- ++\upath;
			\node at (2*\xyd+0.25cm,2*\xyd) {\footnotesize $2$};
		\end{scope}
	\end{tikzpicture}
	\centering
	\caption{Paths associated to the basis elements \eqref{eq:basispaths} in the case $s=1$.}
	\label{fig:basispathsBMW}
\end{figure}
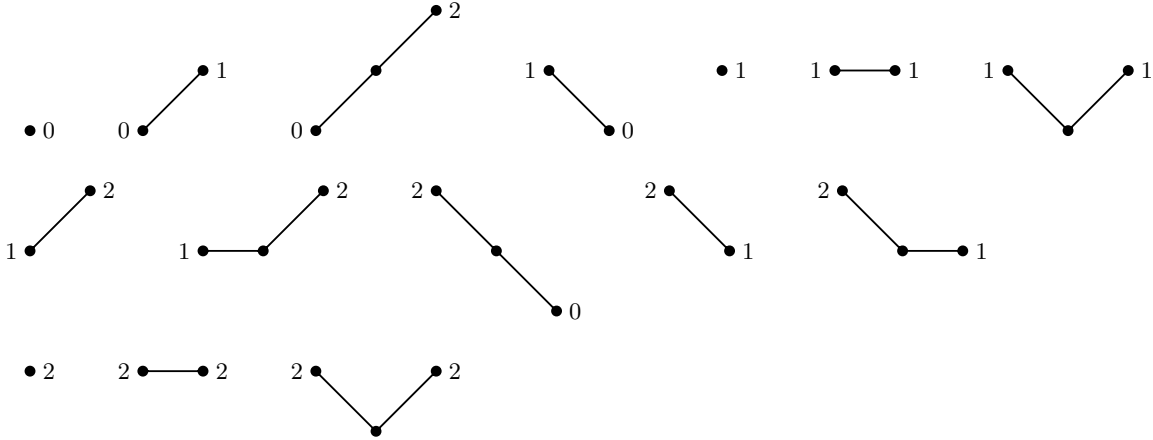
\begin{defi} {\normalfont \cite{IMO}}
	The Birman--Murakami--Wenzl algebra $BMW_3(q,\mu)$ is generated by invertible elements $s_1$ and $s_2$ with the following defining relations
	\begin{align}
		&s_1s_2s_1=s_2s_1s_2, \label{eq:BMWdef1}\\
		&e_1s_1=s_1e_1=\mu^{-1} e_1, \quad e_2s_2=s_2e_2=\mu^{-1} e_2, \label{eq:BMWdef2} \\
		&e_1s_2^{\epsilon}e_1=\mu^{\epsilon}e_1, \quad e_2s_1^{\epsilon}e_2=\mu^{\epsilon}e_2, \quad \epsilon=\pm 1, \label{eq:BMWdef3}\\
		&e_i:=1-\frac{s_i-s_i^{-1}}{q-q^{-1}}, \quad i=1,2. \label{eq:BMWdef4}
	\end{align}
\end{defi}
The centralizer $Z_1 \cong \mathcal{A}_1$ is known to be isomorphic to $BMW_3(q^2,q^4)$. One can verify using the definition \eqref{eq:eir} and the properties \eqref{eq:epropr} that the following relations hold in the algebra $\mathcal{A}_1$
\begin{equation}
	[3]_q e_i^{(0)} = 1 - \frac{q^{-4}\sigma_i-q^4 \sigma_i^{-1}}{q^2-q^{-2}}, \quad i=1,2. \label{eq:e0BMW}
\end{equation}
It is then possible to show that the following mapping from $\mathcal{A}_1$ to $BMW_3(q^2,q^4)$ extends to an algebra isomorphism:
\begin{equation}
	\sigma_i \mapsto q^4 s_i, \quad \text{for } i=1,2. \label{eq:mapBMW}
\end{equation}
Equation \eqref{eq:e0BMW} immediately implies that
\begin{equation}
	e_i^{(0)} \mapsto \frac{1}{[3]_q} e_i, \quad \text{for } i=1,2.
\end{equation}
Moreover, expressing the idempotents $e_i^{(1)},e_i^{(2)}$ in terms of $e_i^{(0)},\sigma_i$ and using \eqref{eq:gi}, one finds 
\begin{equation}
	g_i \mapsto (q+\qi)(q-\qi)^2(s_i-q^{-2}e_i+q^{-1}(q + q^{-1})) + (q+\qi) \quad \text{for } i=1,2. \label{eq:gBMW}
\end{equation}
Note that the mappings \eqref{eq:gTL} and \eqref{eq:gBMW} both agree with the explicit isomorphisms obtained in \cite{CVZ2} relating quotients of the Askey--Wilson algebra with the Temperley--Lieb and Birman--Murakami--Wenzl algebras.
\section{Perspectives \label{eq:conc}}

A presentation in terms of generators and relations of the centralizer of the diagonal action of $\Usl$ in three copies of its spin $s$ irreducible representation was given in this paper.
It has been achieved by combining two well-known algebras: the braid group algebra and the Askey--Wilson one.
We focused on the case with three strands, since it is the first non-trivial case. This gives a good starting point for the general case of $n$-fold tensor products. We know that new relations will be needed for higher $n$ (see for example \cite{CGPV,LZ}) to obtain the centralizer. Nevertheless, one can take our relations and define an algebra on $n$ strands by plugging them on all triples of adjacent strands. This algebra will be larger than the centralizer, but one can still ask whether it is finite-dimensional and study its representation theory for example.
What seems natural for finding the new relations for general $n$ is to bring in the higher rank Askey--Wilson algebra introduced in \cite{PW,BCV,Dec}. In these higher rank cases, the definition of the special quotient of the Askey--Wilson algebra
has not been given yet and deserves more investigation (see \cite{CL} though for $n=4$). 

One may also wonder if results similar to those proven in this paper hold when $q$ is a root of unity. It is well-known that the representation theory of $\Usl$ is more complicated in this case and hence the study of the centralizer should be more involved. For instance, it is known that for $q$ a root of unity there are some degeneracies in the characteristic 
polynomial equation \eqref{eq:chareqsig} for $\sigma_i$ which makes the definition of the idempotents more complicated.
The limits $q\to \pm 1$ are also of interest. Indeed, the case $q\to 1$ would correspond to replacing the Askey--Wilson algebra by the Racah algebra. 
In \cite{CPV}, a quotient of the Racah algebra has been conjectured to be isomorphic to the centralizer of the diagonal action of $U(\gsl_2)$ in any three irreducible representations.
As for $q \to -1$, a similar result has been obtained with the Bannai--Ito algebra and the centralizer of $\osp(1|2)$ in \cite{CVZ,CFV}. The use of the permutation group may simplify the presentations given in these previous papers. We hope to return to some of these questions in the near future.

\paragraph{Acknowledgements.}
NC and LPA thank the CRM for its hospitality and are supported by the international research project AAPT of the CNRS and the ANR Project AHA ANR-18-CE40-0001. The research of LV is supported by a Discovery Grant from the Natural Sciences and Engineering Research Council (NSERC) of Canada. MZ holds an Alexander--Graham--Bell graduate scholarship from NSERC.

\appendix

\section{Proof of Proposition \ref{prop:charkape}}\label{app}
This appendix provides the remaining part of the proof of Proposition \ref{prop:charkape}.

Recall that for some fixed integer $a$ such that $\s<a\leq 2\s$  
\begin{equation}
	\chi_a q_a e_1^{(a)} \sigma_2 e_1^{(a)}-\sum_{r=a,a\pm 1} q_r e_1^{(a)} g_2 e_1^{(r)} \sigma_2 e_1^{(a)} = Q^{(a)}(\kappa) e_1^{(a)}, \label{eq:polkap1}
\end{equation}
where $Q^{(a)}(\kappa)$ is a polynomial in $\kappa$ of degree at most $2\s+1$ that can be computed using the path elements $\x_\gamma$ defined in \eqref{eq:Egam}, with $\gamma \in \Gamma^{(a,a)}$, and using relations \eqref{eq:re2}--\eqref{eq:re4}. The goal is to show that $Q^{(a)}(\kappa)$ does not identically vanish. To do so, it is sufficient to verify that it takes a non-zero value for some well-chosen values of $\kappa$ and $q$ that will simplify the computations.

Note the following useful identity:
\begin{equation}
	\chi_a \pm \chi_b = (q^{a-b} \pm q^{b-a})(q^{a+b+1} \pm q^{-a-b-1}). \label{eq:idchi}
\end{equation}
We start by evaluating $Q^{(a)}(\kappa)$ at
\begin{equation}
	\kappa = \chi_{\s-a}.
\end{equation} 
For any integer $0 \leq b \leq 2\s$, 
\begin{align}
	&\tau_b(\chi_{\s-a})= \frac{\chi_\s(q^{a} + q^{-a})(q^{2\s+1-a} + q^{-2\s-1+a})}{(q^{b} + q^{-b})(q^{b+1} + q^{-b-1})}, \label{eq:tau_eval_kap} \\
	&F_b(\chi_{\s-a}) = \frac{[2\s+1-b]_q[2\s+1+b]_q}{(q^b+q^{-b})^2[2b-1]_q[2b+1]_q}(\chi_{\s-a}-\chi_{\s-b})(\chi_{\s+b}-\chi_{\s-a}). \label{eq:F_eval_kap}
\end{align}
In particular, $F_a(\chi_{\s-a}) = 0$. Consequently, any term in $Q^{(a)}(\chi_{\s-a})$ which is associated to a path $\gamma \in \Gamma^{(a,a)}$ that goes below $a$ vanishes since it is proportional to $F_a(\chi_{\s-a})$, as can be seen diagrammatically from \eqref{eq:creux}. We can hence only consider the functions $\tau_b(\chi_{\s-a})$ and $F_{b+1}(\chi_{\s-a})$ with $b\geq a$. 

We now want to consider the limit
\begin{equation}
	q \to \exp\left( {\frac{2\pi i }{4\s+4}} \right). \label{eq:qlim}
\end{equation}
We must verify that the functions $\tau_b(\chi_{\s-a})$ and $F_{b+1}(\chi_{\s-a})$ with $s<a \leq b\leq 2s$ are non-singular in this limit. 
The factor $(q^b + q^{-b})$, which is present in the denominator of the function $\tau_b(\chi_{\s-a})$, vanishes in this limit only if $b=s+1$. This situation can only occur if $b=a=s+1$. It is seen however that before taking the limit \eqref{eq:qlim}, the factor $(q^b + q^{-b})$ simplifies in \eqref{eq:tau_eval_kap} when $b=a$, giving
\begin{align}
	\tau_a(\chi_{\s-a})= \frac{\chi_\s(q^{2\s+1-a} + q^{-2\s-1+a})}{q^{a+1} + q^{-a-1}}.
\end{align}
The function $F_{b+1}(\chi_{\s-a})$ contains factors $(q^{b+1} + q^{-b-1})$ in the denominator which can never vanish in the limit \eqref{eq:qlim} for $b\geq a >s$. There is also a factor $[2b+1]_q$ in the denominator of $F_{b+1}(\chi_{\s-a})$ that can vanish only if $b=s+1/2$. This can occur only if $b=a=s+1/2$. Once again, it is seen from \eqref{eq:F_eval_kap} that the problematic factor simplifies before the limit \eqref{eq:qlim} is taken, giving
\begin{align}
	F_{a+1}(\chi_{\s-a}) = \frac{[2\s-a]_q[2\s+2+a]_q}{(q^{a+1}+q^{-a-1})^2[2a+3]_q}(q - q^{-1})^2(q^{2\s-2a} - q^{2a-2\s})(q^{2\s+2} - q^{-2\s-2}). \label{eq:F_eval_kap2}
\end{align}
Finally, there is a factor $[2b+3]_q$ in the denominator of $F_{b+1}(\chi_{\s-a})$ that never vanishes when $b>s$. Therefore, we can apply the limit \eqref{eq:qlim} to the functions $\tau_b(\chi_{\s-a})$ and $F_{b+1}(\chi_{\s-a})$. 

Recall that
\begin{equation}
	\sigma_2 = \sum_{r=0}^{2\s} q_r e_2^{(r)}
\end{equation}
and
\begin{equation}
	e_2^{(r)}=\prod_{\genfrac{}{}{0pt}{}{p=0}{p\neq r}}^{2\s} \frac{g_2-\chi_p}{\chi_r-\chi_p}. \label{eq:e2g2}
\end{equation}
In the limit \eqref{eq:qlim}, the factors $\chi_r-\chi_p$ vanish only if $p=r$ or $p=2s+1-r$, as can be seen using the identity \eqref{eq:idchi}. Therefore, the idempotent $e_2^{(0)}$ written as in \eqref{eq:e2g2} is always non-singular in this limit. For $0 \leq r < 2\s$, the idempotent $e_2^{(r+1)}$ has a factor $\chi_{r+1}-\chi_{2\s-r}$ in the denominator that vanishes, except when $s$ is half-integer and $r=s-1/2$. Note also that $q_{2\s-r}=q_{r+1}$ in the limit \eqref{eq:qlim}. If $\s$ is an integer, one can write
\begin{align}
	\sigma_2 
	&= e_2^{(0)} + \sum_{r=0}^{s-1}\left(q_{r+1} e_2^{(r+1)} + q_{2\s-r} e_2^{(2\s-r)}\right)\\
	&= e_2^{(0)} + \sum_{r=0}^{s-1}q_{r+1} \left( e_2^{(r+1)} + e_2^{(2\s-r)}\right) + \sum_{r=0}^{s-1}(q_{2\s-r}-q_{r+1}) e_2^{(2\s-r)}. \label{eq:sig2lim1}
\end{align}
Similarly, if $\s$ is half-integer one has
\begin{align}
	\sigma_2 
	&= e_2^{(0)} + q_{s+\frac{1}{2}}e_2^{(s+\frac{1}{2})}  + \sum_{r=0}^{s-\frac{3}{2}}q_{r+1} \left( e_2^{(r+1)} + e_2^{(2\s-r)}\right) + \sum_{r=0}^{s-\frac{3}{2}}(q_{2\s-r}-q_{r+1}) e_2^{(2\s-r)}. \label{eq:sig2lim2}
\end{align}

For some fixed integer $r$ such that $0 \leq r < 2\s$ and $r\neq s-\frac{1}{2}$, we define the polynomial 
\begin{equation}
	P(X) := \prod_{\genfrac{}{}{0pt}{}{p=0}{p\neq r+1,2\s-r}}^{2\s} (X-\chi_p). \label{eq:PX}
\end{equation}
Using \eqref{eq:e2g2} and \eqref{eq:PX}, one can write
\begin{align}
	e_2^{(r+1)} + e_2^{(2\s-r)} 
	&= \frac{-P(g_2)}{P(\chi_{r+1})P(\chi_{2\s-r})}\left( \frac{(g_2-\chi_{r+1})P(\chi_{r+1})-(g_2-\chi_{2\s-r})P(\chi_{2\s-r})}{\chi_{r+1}-\chi_{2\s-r}} \right). 
\end{align}
The prefactor is defined in the limit \eqref{eq:qlim}. We now expand the parenthesis in $(\chi_{r+1}-\chi_{2\s-r})$ and up to terms proportional to $(\chi_{r+1}-\chi_{2\s-r})$, we find:
\begin{equation}
	e_2^{(r+1)} + e_2^{(2\s-r)}= \frac{-P(g_2)}{P(\chi_{r+1})P(\chi_{2\s-r})}\left( (g_2-\chi_{2\s-r})P'(\chi_{2\s-r}) - P(\chi_{2\s-r}) + \dots \right), \label{eq:eterm1}
\end{equation}
which has no singularities in the limit \eqref{eq:qlim}.

The remaining term to be considered is
\begin{equation}
	 (q_{2\s-r}-q_{r+1})e_2^{(2\s-r)} = \frac{(q_{2\s-r}-q_{r+1})}{(\chi_{2\s-r}-\chi_{r+1})}\frac{P(g_2)}{P(\chi_{2\s-r})}(g_2-\chi_{r+1}). \label{eq:eterm2}
\end{equation}
Using the identity \eqref{eq:idchi}, one can write
\begin{equation}
	(\chi_{2\s-r}-\chi_{r+1}) = (q^{2\s-2r-1}-q^{-2\s+2r+1})(q^{2\s+2}-q^{-2\s-2}).
\end{equation}
This factor which appears in the denominator of \eqref{eq:eterm2} has only a simple root at the value of the limit \eqref{eq:qlim} since we have supposed $r\neq s-\frac{1}{2}$. The factor $(q_{2\s-r}-q_{r+1})$ vanishes in this limit and hence also has a root at this value of $q$. It follows that the term \eqref{eq:eterm2} is not singular.

The previous discussion implies that it is justified to express $\sigma_2$ as in \eqref{eq:sig2lim1} or \eqref{eq:sig2lim2} and then take the limit \eqref{eq:qlim}. One finds from \eqref{eq:F_eval_kap2} that $F_{a+1}(\chi_{\s-a})\to 0$ in this limit. It follows that any term in the polynomial $Q^{(a)}(\chi_{\s-a})$ that is associated to a path $\gamma \in \Gamma^{(a,a)}$ that goes above $a$ vanishes, as can be seen from \eqref{eq:pic}. Hence, we need only consider the terms corresponding to the constant paths in $\Gamma^{(a,a)}$. One can also verify that $\tau_a(\chi_{\s-a}) \to \chi_0$ in the limit \eqref{eq:qlim}. The LHS of \eqref{eq:polkap1} thus becomes
\begin{equation}
	q_a(\chi_a - \chi_0) e_1^{(a)} \sigma_2 e_1^{(a)}, \label{eq:polkap2}
\end{equation} 
where $\sigma_2$ is written as in \eqref{eq:sig2lim1} or \eqref{eq:sig2lim2}, and $g_2$ is replaced by $\chi_0$. Since $P(\chi_0)=0$, it is seen from \eqref{eq:eterm1}, \eqref{eq:eterm2} and \eqref{eq:e2g2} that all the terms in \eqref{eq:sig2lim1} or \eqref{eq:sig2lim2} vanish except for $e_1^{(0)}$ which becomes $1$. Hence we can replace $\sigma_2 $ by $1$ in \eqref{eq:polkap2}. We finally deduce that 
\begin{equation}
	Q^{(a)}(\chi_{\s-a}) \to q_a(\chi_a - \chi_0),
\end{equation}
which does not vanish since $0 \leq s<a\leq 2s$. We conclude that $Q^{(a)}(\kappa)$ is not identically zero.

\end{document}